\documentclass[10pt]{article}

\usepackage{amsmath,amsthm}
\usepackage{amsfonts,amssymb}
\usepackage{graphicx}
\usepackage{multirow}
\usepackage[authoryear]{natbib}
\usepackage{setspace}
\usepackage[final]{showkeys}
\usepackage{microtype}
\usepackage{fouriernc}
\usepackage[hyphens]{url}
\usepackage{fix-cm}

\allowdisplaybreaks

\newcommand{\R}{\mathbb{R}}
\newcommand{\p}{\mathbb{P}}
\newcommand{\e}{\mathbb{E}}

\newtheorem{theo}{Theorem}
\newtheorem{lemma}{Lemma}

\newtheorem{example}{Example}
\newtheorem{nb}{Remark}
\newtheorem{algo}{Algorithm}

\begin{document}

	\title{On Firing Rate Estimation for Dependent Interspike Intervals}
	
	\author{Elisa Benedetto, Federico Polito, Laura Sacerdote\footnote{Corresponding author. Email: laura.sacerdote@unito.it}\\
		Department of Mathematics \emph{G. Peano}, University of Torino\\
		Via Carlo Alberto 10, 10123, Torino, Italy}
	
	\date{}
	
	\maketitle

	\begin{abstract}

		If interspike intervals are dependent the instantaneous firing rate does not catch important features of spike trains.
		In this case the conditional instantaneous rate plays the role of the instantaneous firing rate for the case of samples of independent
		interspike intervals. If the conditional distribution of the interspikes intervals is unknown, it becomes difficult to evaluate the
		conditional firing rate. 
		We propose a non parametric estimator for the conditional
		instantaneous firing rate for Markov, stationary and ergodic ISIs. An
		algorithm to check the reliability of the proposed estimator is introduced
		and its consistency properties are proved. The method is
		applied to data obtained from a stochastic two compartment model and to in vitro
		experimental data.

		\medskip		
		
		\noindent \textit{Keywords:} Firing rate; Non parametric estimator; Markov ISIs.
	
	\end{abstract}

	\section{Introduction}

		The firing rate characterizes the behavior of spiking neurons and it is often
		used to evaluate the information encoded in neural responses. Its use dates
		back to \citet{Adrian} who first assumed that information sent by a neuron is
		encoded in the number of spikes per observation time window. The frequency
		transfer function is then often studied to illustrate the relationship
		between input signal and firing rate, taken as constant for a fixed signal.
		Typically, the firing rate increases, generally non-linearly, with an increasing
		stimulus intensity \citep{Kandel}. The main objection to the use of
		the transfer function is that it ignores any information possibly encoded in
		the temporal structure of the spike train. Furthermore, different definitions of \emph{rate}
		give rise to different shapes of the frequency transfer function.
		They are equivalent under specific
		conditions and differ in other instances (see \citet{LRS}). 

		When the firing rate is not constant,
		specific techniques are necessary to estimate the instantaneous
		firing rate avoiding artifacts (see \citet{cun} for a review). An
		important tool in this framework is the kernel method. Optimizing the choice of
		kernels and of bins becomes necessary in the
		case of a time dependent input (see \citet{shim}, \citet{omi}).
		A method to estimate the firing rate in presence of
		non-stationary inputs was recently proposed in \citet{Iolov} for
		the case of slowly varying signals while \citet{Tamborrino} consider also the
		latency to improve the firing rate estimation. 

		A time dependent instantaneous firing rate is often related to a time
		dependent input 
		but in fact appears anytime the ISIs are not exponentially distributed. The case of independent identically distributed ISIs can be studied
		in the framework of renewal theory for point processes (see e.g.\ \citet{co}). However, neuronal spike trains rarely verify the renewal
		hypothesis. Even those spike trains that can be described as stationary often exhibit serial dependence among interspike intervals as already
		noted in \citet{{Perkel1}}, \citet{John} and more recently showed in \citet{fark}. The presence of an external signal is a further cause of
		dependence between ISIs. In this case, beside the renewal property, the process may lose also the stationarity. For example, adaptivity in
		presence of an external input determines non stationary dependent ISIs. \citet{fark2} enlighten important features of neural code determined
		by this phenomenon: decrease of the noise in the population rate code, faster and more reliable transmission of modulated input rate signals
		and the fact that a postsynaptic neuron can better decode a small change in its input if the presynaptic neurons are adaptive.
		\citet{muller} propose a Markov model to describe adaptivity. Their study is performed on a group of neurons by considering the ensemble
		firing rate in the framework of inhomogeneous renewal processes theory through inhomogeneous hazard rate functions.
 
		Dependent ISIs are observed also in absence of external input.
		\citet{fark} report experimental findings
		of negative serial correlation in ISIs observed in the sensory periphery as
		well as in central neurons in spontaneous activity.
		These observations cannot be related to
		input features and arise in stationary conditions. Poisson and gamma renewal processes are
		used extensively to model ISIs of cortical neurons (see for example \citet{vanrossum,song,rudolph,shelley}).
		For these neurons the instantaneous firing rate is generally estimated assuming the independence of ISIs.
		However, when the ISIs are dependent, the instantaneous firing rate is no longer the correct tool to
		analyze the data. Indeed it does not change for independent or dependent ISIs when the common marginal distribution is the same
		(think for example of exponentially or gamma distributed but dependent ISIs).		
		The use of the hypothesis of stationarity allows us to introduce the necessary statistical tools for the study of identically
		distributed dependent ISIs.
		
		To enlighten the role of the ISIs dependencies on the firing rate, a conditional
		instantaneous firing rate is considered. Indeed by using conditional
		distributions it is possible to remove the classical renewal hypothesis
		\citep{co}.
		The instantaneous conditional firing is updated after each ISI, thus accounting
		for the dependency between ISIs. \citet{John} computed the conditional firing time of ISIs following an ISI of specified duration. 

		Estimators for the conditional spiking intensity exist in the literature in the
		framework of maximum likelihood methods
		(see \citet{Brown,Gerhard,Liu}.
		Advantages of such estimators include the possibility to perform tests of
		hypotheses, to evaluate efficiency features or to compute standard errors.
		However, these advantages are based on the assumption that a specific model
		fits the data and no result holds when the model is not true or is unknown.
		Unfortunately, the model is known in a relative small number of cases and
		often its choice introduces non controllable approximations.
		Regression modeling of conditional intensity was considered by \citet{Brillinger}.
		Recent works in this framework concern seismologic analysis and perform non parametric regression using kernel methods
		to consider space-time dependencies \citep{Zhuang,adelfio}.
		
		Aim of this paper is to determine a conditional firing rate estimator for
		dependent ISIs avoiding to introduce a model to describe their distribution.
		We suggest a non parametric estimator for the instantaneous firing rate,
		conditioned on the past history of the spike train.
		Under suitable hypotheses we prove the consistency of the
		proposed estimator.

		The paper is organized in the following way. In Section \ref{S2} we recall
		different definitions of the firing rate taking care to enlighten
		their features if the ISIs are dependent. In Section \ref{S3} we propose an
		estimator of the conditional instantaneous firing rate of a neuron
		characterized by an ISI sequence that can be modeled as a Markov, ergodic
		and stationary process. The consistency of this estimator is proved in
		the Appendix. The direct check of Markov, ergodic and stationary hypotheses on
		the available data is not easy but it is necessary to make reliable the
		estimates. Hence in Section \ref{S5} we develop a statistical algorithm to
		validate \emph{a posteriori} our estimation by checking the underlying hypotheses
		by means of the obtained estimate. In Section \ref{S6} we apply the proposed method
		to estimate the conditional firing rate of simulated dependent ISIs. For this aim we
		resort to the stochastic two compartment model proposed in \citet{Lansky}.
		Indeed \citet{BenSac} recently showed that suitable choices of the
		parameters of this simple model allow to generate sequences of ISIs that
		are statistically Markov, ergodic and stationary. Finally, in Section \ref{S7} we
		apply our results to in vitro experimental data.

	\section{Definitions of the firing rate}
	
		\label{S2}
		The firing rate can be defined in many different ways. Here we report the most usual ones, their relationships and which spike
		train features best enlighten each of them. The ability of each  definition to catch the nature of neural code strongly depends on the nature
		of the observed spike train. Specifically we distinguish the case of independent identically distributed ISIs from the case with dependent ISIs.
		
		Let us describe the spike train through the ISIs and let $T$ be a random variable with the same distribution of the ISIs, shared by the
		entire spike train. The most used definitions \citep[see among others][]{Burkitt,Rullen} are 
		the inverse of the average ISI
		\begin{align}
			\label{fr}
			1/\e(T),
		\end{align}
		which is commonly called \emph{firing rate}, and
		the expectation of the inverse of an ISI
		\begin{align}
			\label{inv}
			\e(1/T),
		\end{align}
		called the \emph{instantaneous mean firing rate}.
		Values obtained with the second definition are always larger than those computed on the considered interval
		through $1/\e(T)$ \citep{LRS}.
		Identical distribution of ISIs is required to give sense to these definitions. 
	
		A different description of spike trains makes use of the number of spikes occurred up to time $t$, i.e.\ the counting process
		$N(t)$, $t \ge 0$ \citep[see e.g.][]{Cox_libro}. The quantity $N(t)/t$, is a measure of the instantaneous firing rate accounting for the total
		number of spikes $N(t)$ up to time $t$.
		The firing rate \eqref{fr} and the quantity $N(t)/t$ are connected through the well-known asymptotic formula
		\begin{equation}
			\label{22}
			\frac{1}{\e(T)} = \lim_{t\rightarrow\infty}\frac{\e(N(t))}{t}
		\end{equation}
		\citep{co,Rudd}.
		An alternative definition of the instantaneous firing rate is \citep[e.g.][]{Johnson}
		\begin{equation}
			\label{TFR}
			\lambda(t) = \lim_{\Delta t\rightarrow 0}\frac{\e[N(t+\Delta t)-N(t)]}{\Delta t}.
		\end{equation}   
		This last expression does not consider the cumulative property of $N(t)/t$ that involves the whole history of the spike train.
		When the ISIs have a common distribution, it can be expressed in terms of the ISIs distribution through 
		\begin{equation}
			\label{ITFR}
			\lambda(t) = \frac{f(t)}{\p(T>t)},
		\end{equation} 
		where $f(t)$ is the probability density function of the random ISI $T$ \citep[][Chapter 6]{CL}.
		
		\begin{example}[Exponential ISIs with refractory period]
			To illustrate the use of these quantities let us consider a spike train with exponentially distributed ISIs, which cannot be shorter
			than a constant $\delta$. This constant $\delta$ models the refractory period of a neuron \citep{LRS}.
			Then the ISI distribution is $F(t)=\p(T\leq t) = 1-\exp\lbrace-(t-\delta)\rbrace$, for every $t\geq\delta$.
			In this case the firing rate \eqref{fr} and
			the mean instantaneous firing rate \eqref{inv} are, respectively,
			\begin{align}
				\frac{1}{\e(T)} & = \left( \int_{\delta}^{\infty} te^{-(t-\delta)}dt \right)^{-1} = \frac{1}{\delta + 1}, \\
				\e\left(1/T\right) & = \int_{\delta}^{\infty} t^{-1}e^{-(t-\delta)}dt = e^\delta \Gamma(0,\delta) < \infty, \qquad \forall \, \delta > 0,
			\end{align}
			where $\Gamma(a,b)$ is the incomplete Gamma function.
			On the other hand if we compute the instantaneous firing rate by \eqref{TFR} or, equivalently, by \eqref{ITFR}, we get
			\begin{equation}
				\label{ExFR}
				\lambda(t) = 
					\begin{cases}
						0, & 0<t<\delta, \\
						1, & t\ge \delta.
					\end{cases}
			\end{equation}
			We explicitly remark that we would obtain the same firing rates in the case of dependent ISIs with the same common distribution,
			as \eqref{fr}, \eqref{inv} and \eqref{ITFR} do not consider any possible dependence on spike train history. 
		\end{example}
		
		More realistic instances should consider either dependent and/or not identically distributed random variables.
		In particular, defining the firing rate as the
		instantaneous firing rate (\ref{TFR}), conditioned on the spike train history 
		\citep{Snyder},
		\begin{equation}
			\label{CTFR}
			\lambda^*(t)=\lambda(t|N(s), 0<s\le t ) =\lim_{\Delta t\rightarrow 0}\frac{\e\left[N(t+\Delta t)-N(t)|N(s), 0<s\le t\right]}{\Delta t},
		\end{equation}
		makes possible to account for ISIs dependencies.
		Mathematically, function (\ref{CTFR}) is known as the conditional intensity function of the underlying counting process $N(t)$.
		It measures the probability of a spike in $[t,t+\delta t)$ when the presence of a spike depends on the history of the spike train.
		If the process is an inhomogeneous Poisson process, $\lambda^*(t)$ coincides with the rate of the process.
		It can also be defined in terms of the ISIs $T_i,i \geq 1$ given the past sequence of the ISIs $\{T_{j}\}_{j=1,...,(i-1)}$.
		Consider a fixed time interval $[0,L]$, $0<L<\infty$ and  denote by $l_1, l_2, \ldots, l_{N(L)}$ the ordered set of the
		firing epochs until time $L$.
		Then the ISIs sequence can be expressed as $\lbrace T_i = l_i-l_{i-1}, i\geq 1,l_0=0 \rbrace$ and the conditional probability density is
		\begin{equation}
			\label{hazard2b}
			f_i(t| T_j, j=1,\dots, i-1) = \frac{d}{dt} \p (T_i\leq t|T_j, j=1,\dots, i-1).
		\end{equation} 
		It is now possible to introduce the alternative definition \citep[][Section 7.2]{Daley}, 
		\begin{equation}
			\label{lambda}
			\lambda^*(t)= 
			\begin{cases}
			     h(t), & 0<t\leq l_1, \\
			     h_i(t-l_{i-1}|T_j, j=1,\dots, i-1), & l_{i-1}<t\leq l_{i}, \: i\geq 2.
			\end{cases}
		\end{equation}
		This definition can be proven equivalent to (\ref{CTFR}). Here the functions 
		\begin{subequations}
			\begin{equation}
				\label{hazard1a}
				h(t)=-\frac{d}{dt}\ln\left[S(t)\right]=\frac{f(t)}{S(t)},
			\end{equation}
			\begin{align}
				\label{hazard1b}
				h_i(t|T_j, j=1,\dots, i-1) & =-\frac{d}{dt}\ln\left[S_i(t|T_j, j=1,\dots, i-1)\right] \notag \\
				& =\frac{f_i(t| T_j, j=1,\dots, i-1)}{S_i(t|T_j, j=1,\dots, i-1)},
				\qquad i\geq2,
			\end{align}
		\end{subequations}
		are the ISI hazard rate function and the conditional ISI hazard rate function, respectively.
		Furthermore, $S(t)=1-\p(T\leq t)=1-\int_0^t f(s)\,ds $ and
		$S_i(t|T_j, j=1,\dots, i-1)=1-\p(T_i\leq t|T_j, j=1,\dots, i-1)
		=1-\int_0^t f_i(s|T_j, j=1,\dots, i-1)\,ds$ are the corresponding survival functions.
		
		Next in order of complexity to renewal processes, are point processes characterized by successive inter-events intervals that
		form a first order Markov chain. 
		For a renewal process the instantaneous firing rate depends only on the time elapsed since the last spike occurred, while for a first order
		Markov point process it depends also on the previous ISI.
		In the following we determine the conditional instantaneous firing rate for a model with this property.
		Note that the idea to consider inter-events models according with a first order Markov chain is not new. A first example
		in this direction is discussed in \citep{Lamp} where a marked Poisson process characterized by Markov inter-event times is discussed.
		Other classical examples of first order Markov inter-event times for point processes are discussed in \citet{Law}. Here we
		introduce an alternative example illustrating also the use of copulas to generate inter-event times with assigned marginal
		distribution that determine a Markov chain.
	
		\begin{example}[Exponential ISIs with refractory period II]
			\label{ued}
			Let us reconsider the example with exponential ISIs in presence of a refractory period. It is obvious that the conditional
			instantaneous firing
			rate coincides with the unconditional instantaneous firing rate (\ref{ExFR}) when we have independent ISIs.
			 
			Then let us assume the presence of dependence between ISIs.
			Avoiding any attempt to give a biological significance to our example let us here hypothesize that
			\begin{itemize}
				\item the ISI sequence is Markov, i.e.\ in \eqref{lambda}, we have $h_i(t-l_{i-1}|T_j,j=1,\dots,i-1) = h_i(t-l_{i-1}|T_{i-1})$,
				\item the ISIs $T_i$ and $T_{i+1}$, $i\geq 1$, are dependent with joint distribution function
				\begin{align*}
					\p(T_i\leq \tau,T_{i+1}\leq \vartheta)=F(\tau)F(\vartheta)\lbrace 1+[1-F(\tau)][1-F(\vartheta)]\rbrace,
				\end{align*}
				where $\tau,\vartheta>0$, $F(t)=\p(T\le t)$,
				that corresponds to considering a copula distribution of the Farlie--Gumbel--Morgenstern family
				(see \citet{Nelsen}, formula 3.2.10),
				\begin{align*}
					C_{\alpha}(u,v) = uv\left(1+\alpha (1-u)(1-v)\right),
				\end{align*}
				with $\alpha=1$.
			\end{itemize}
			If we compute the conditional instantaneous firing rate (\ref{CTFR}), according to equation (\ref{lambda}), we get
			$\lambda^*(t)=1$ for $\delta<t\leq l_1$, $\lambda^*(t)=0$, for $0<t\le \delta$, and 
			\begin{equation}
				\label{norn}
				\lambda^*(t)=
				\begin{cases}
					0, & l_{i-1} < t \le l_{i-1}+\delta, \\
					\frac{1+(1-2e^{-(t-l_{i-1}-\delta)})(1-2e^{-(T_{i-1}-\delta)})}
					{2-e^{-(t-l_{i-1}-\delta)}-2e^{-(T_{i-1}-\delta)}+2e^{-(t-l_{i-1}-2\delta +T_{i-1})}}, & l_{i-1}+\delta <t\leq l_{i},
				\end{cases}
			\end{equation}
			with $i\ge 2$.
			Despite the exponential distribution of the ISIs,
			we note that dependent ISIs determine a time dependent instantaneous firing rate (see Fig.\ \ref{mouse} for an example;
			in grey, for comparison, the instantaneous firing rate \eqref{TFR}).
			\begin{figure}
				\centering
				\includegraphics[scale=0.68]{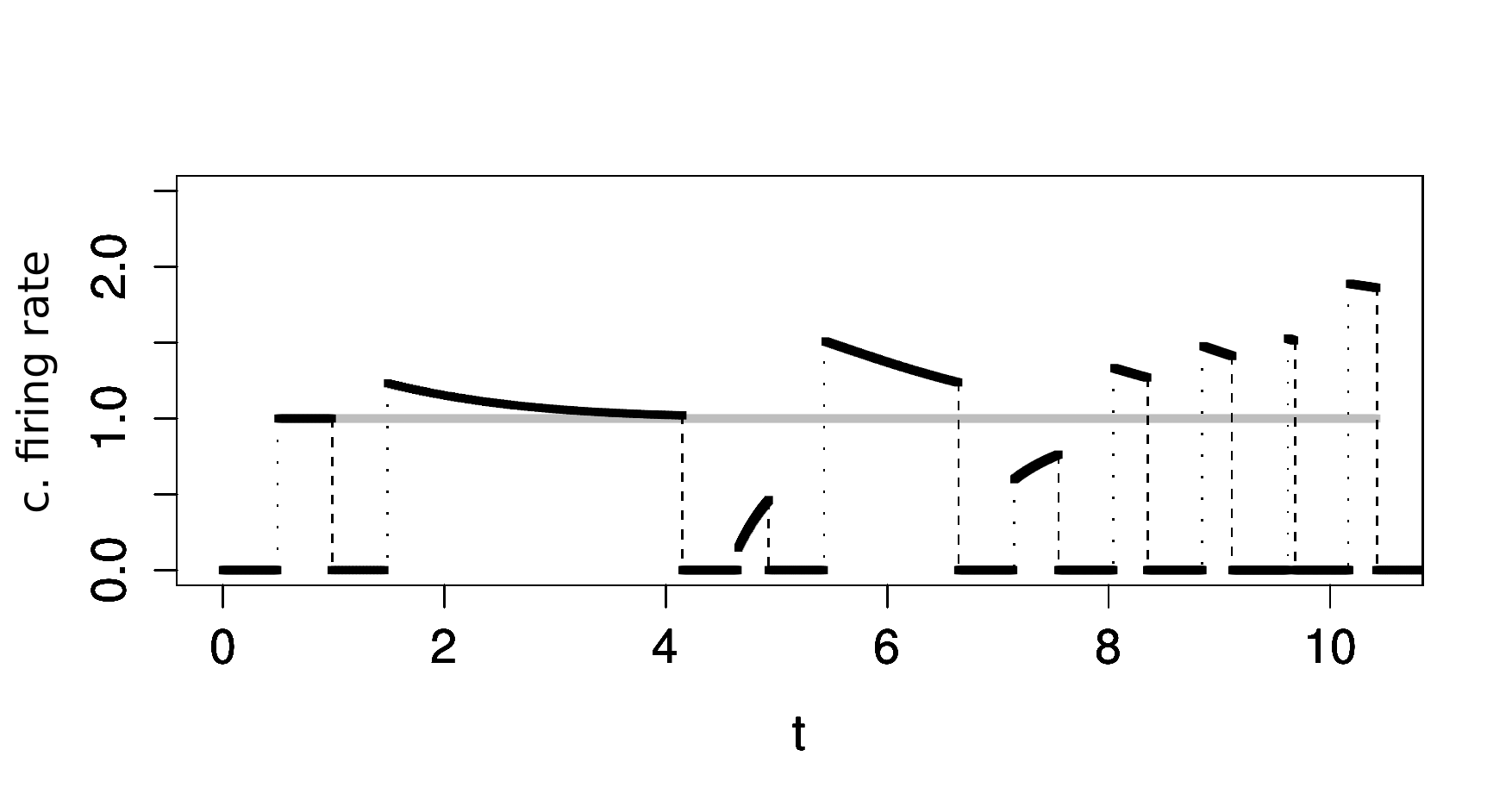}
				\caption{\label{mouse}A realization of the theoretical conditional instantaneous firing rate \eqref{norn} (black) together
				with the corresponding instantaneous firing rate \eqref{ExFR} (grey). Here the parameters are fixed at $(\lambda,\delta)
				=(1,0.5)$ and the recorded ISIs are indicated by successive dashed vertical lines.
				Note that the conditional instantaneous firing rate is \emph{updated} after each recorded spike.
				Except the initial period of time after each spike where the conditional instantaneous firing rate is zero
				(due to presence of refractoriness through the parameter $\delta$), in general
				it is not constant. Note that the unconditional and conditional firing rates coincide for $0<t<\delta$.
				In order to clarify the contribution of the dependence between consecutive ISIs to the
				behavior of the conditional instantaneous firing rate, note that after a very short ISI (such as the one starting right after time
				$t=4$) the conditional instantaneous firing rate becomes very high (after the time lapse $\delta$ during which is zero).
				Viceversa, the ISI recorded around time $t=6$, which is quite long, implies a succeeding low conditional
				instantaneous firing rate. Furthermore notice how the conditional instantaneous firing rate exhibits a relaxation
				towards the instantaneous firing rate (grey) as time passes (see for instance the ISI starting around $t=1$). This
				is connected to the fact that memory (dependence) is less and less important as time flows.}
			\end{figure}
		\end{example}
		
		The estimation of the firing rate makes use of the observed ISIs or of the observed counting process. Usually the estimators
		\begin{align}
			1/\overline{T}&=n/\sum_{i=1}^n T_i, \label{FR} \\
			\overline{1/T}&=\frac{1}{n}\sum_{i=1}^n \frac{1}{T_i}, \label{IFR}
		\end{align}
		are used for the inverse of the average ISI and for the instantaneous mean firing rate, respectively. The ratio between the number of spikes in
		$(0,t)$	and $t$ can be used to estimate $N(t)/t$. While the considered estimators have a very intuitive meaning if the ISIs
		are independent and identically distributed they lose their significance for dependent or not identically distributed ISIs. 
		When the process is not stationary these estimators are applied on successive suitable time windows in order to detect
		the time evolution of the observed firing rate.
	
		Estimation techniques for equations \eqref{ITFR} and \eqref{lambda} depend upon the knowledge of $f(t)$ and $f_i(t| T_j, j=1,\dots, i-1)$,
		$i=2,3,\dots$, respectively.
		When these densities are given or hypothesized from data, classical parametric methods can be applied \citep{kalb}.
		Hazard rate function estimators assuming the independence of the sample random variables have good convergence properties,
		but their application to dependent neural ISIs is not legitimate. A contribution to the study of hazard rate function estimation
		in presence of dependence is considered in \citet{muller,fark2} in the framework of adaptation analysis.
		In the same framework, \citet{naud} discuss a quasi-renewal approach to conditional firing rate estimation (see also the references therein).
		A contribution to the study of the stationary case is the subject of the following sections.

	\section{Firing rate estimation}

		\label{S3}
		The definition (\ref{lambda}) of the conditional instantaneous firing rate depends on the hazard rate functions (\ref{hazard1a})
		and (\ref{hazard1b}) of the underlying ISI process. Hence our firing rate estimation problem is strictly connected to the estimation
		of these hazard rate functions.
		
		We model the ISIs as a Markov, ergodic and stationary process. Hence the ISIs are identically distributed random variables with shared
		unconditional probability density function $f(t)$. Furthermore, due to the Markov hypothesis, the conditional probability density functions
		(\ref{hazard2b}) become 
		\begin{equation}
			\label{16}
			f_i(t|T_{i-1}=\tau, T_j=t_j, j=1,\dots,i-2) = f_i(t| T_{i-1}=\tau)=f(t| \tau), \qquad \forall \, i\geq 1.
		\end{equation}
		Therefore, the conditional hazard rate function (\ref{hazard1b}) simplifies to
		\begin{align}
			h(t| \tau)&=\frac{f(t| \tau)}{S(t| \tau)}, \label{h2}
		\end{align}
		where $S(t| \tau)=1-\int_0^t f(s|\tau)\,ds$ is the associated conditional survival function.

		In order to estimate the instantaneous conditional firing rate \eqref{lambda} we propose the following non-parametric
		estimator for the ISI conditional hazard rate function \eqref{h2}:
		\begin{equation}
			\label{hazardc}
			\hat{h}_n(t|\tau)=\frac{\hat{f}_n(t|\tau)}{\hat{S}_n(t| \tau)},
		\end{equation}
		where $\hat{S}_n(t|\tau)=1-\int_0^t \hat{f}_n(s|\tau)\,ds$ and \citep{Arfi,Ould-Said} 
		\begin{equation}
			\label{cond-estimator}
			\hat{f}_n(t|\tau)=\frac{\hat{f}_n(\tau,t)}{\hat{f}_n(\tau)},
		\end{equation}
		with
		\begin{align}
			\label{kernel2}
			\hat{f}_n(\tau,t)=\frac{1}{nc_n^{2}}
			\sum_{i=1}^n K_1\left(\frac{\tau-T_i}{c_n}\right)K_2\left(\frac{t-T_{i+1}}{c_n}\right),
		\end{align}
		\begin{align}
			\label{kernel}
			\hat{f}_n(t)=\frac{1}{nc_n}
			\sum_{i=1}^n K_1\left(\frac{\tau-T_i}{c_n}\right).
		\end{align}
		The estimator \eqref{cond-estimator}
		is a uniform strong consistent estimator for the ISI conditional probability density function \eqref{hazard2b} on any compact interval
		$[0,J] \subseteq \R_+$.
		Here $\lbrace c_n \rbrace$ is a sequence of positive real numbers satisfying
		\begin{equation}
			\label{1}
			\lim_{n\rightarrow \infty} c_n=0 ,\qquad \qquad 
			\lim_{n\rightarrow \infty}nc_n=\infty,
		\end{equation} 		
		and $K_j$, $j=1,2$, are real-valued kernel
		functions verifying
		\begin{equation}
			\label{3}
			K_j(t)\geq 0, \qquad \int_{-\infty}^{\infty}K_j(t) \, dt = 1, \qquad
			\lim_{|t|\rightarrow \infty}|t| \, K_j(t)=0.
		\end{equation}
		Moreover, we assume that the kernels have bounded variation and that $K_1$ is strictly positive.

		Note that, in order to estimate, the use of a sample of ${T_i}$ implies to start the observation in correspondence of a spike time,
		i.e. an arbitrary event. Hence, according with \citet{Law} we perform a \textit{synchronous sampling}.
		A different approach could consider an arbitrary starting time; relations between these two cases were discussed in
		\citet{mcfadden}.

		\begin{nb}		
			We do not consider the value of the instantaneous firing rate on the first ISI because of the lack of knowledge
			on the preceding ISI. However, it is possible to estimate the unconditional firing rate through
			\begin{equation}
				\label{hazardl}
				\hat{h}_n(t)=\frac{\hat{f}_n(t)}{\hat{S}_n(t)},
			\end{equation}
			where $\hat{S}_n(t)=1-\int_0^t \hat{f}_n(s)\,ds$.
		\end{nb}

		In Figure \ref{figfig} we reconsider Example \ref{ued} for which we compare the theoretical conditional instantaneous firing rate \eqref{norn}
		with its conditional estimator \eqref{hazardc}.
		\begin{figure}
			\centering
			\includegraphics[scale=.6]{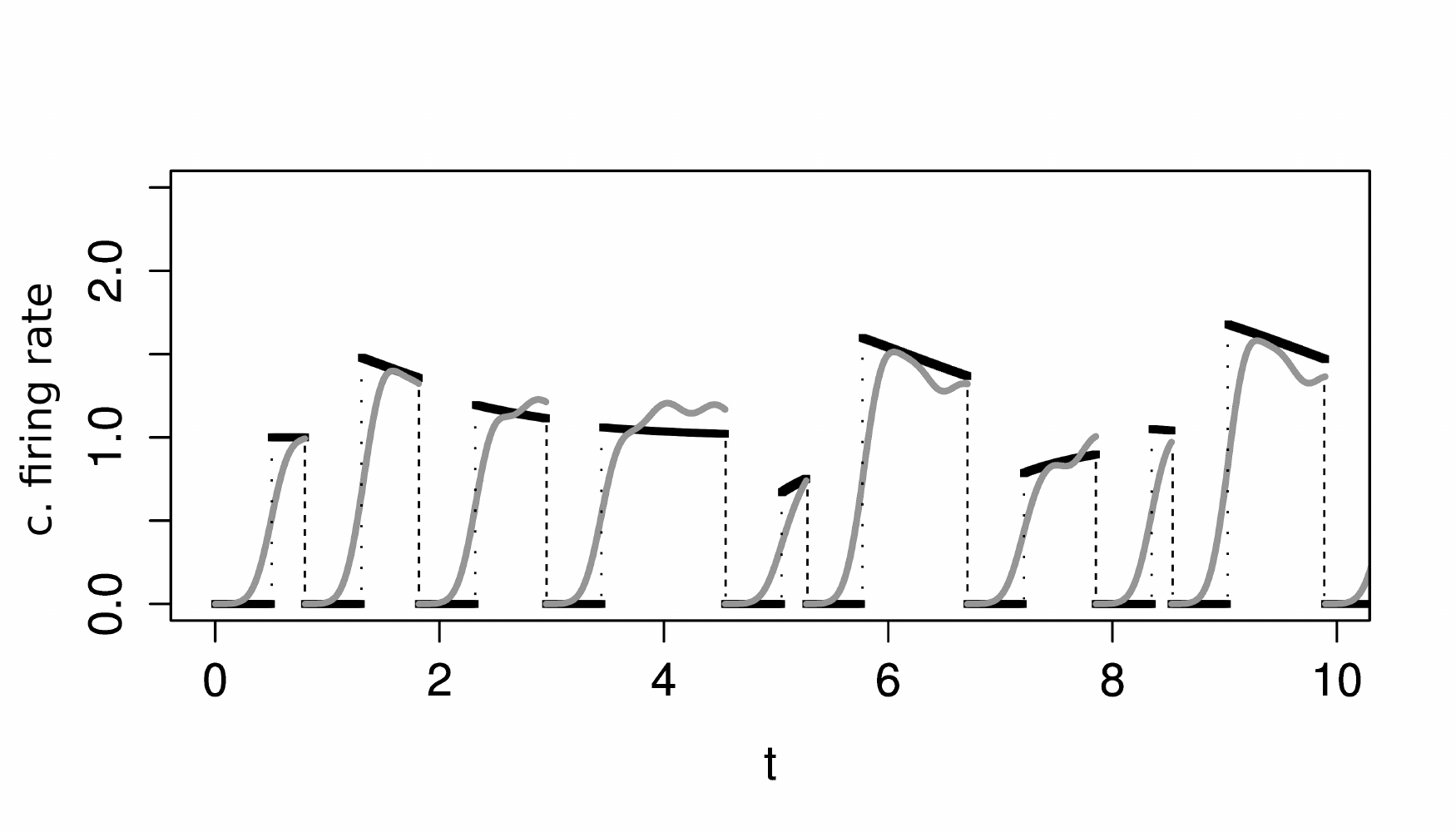} \\
			\vspace{-1.3cm}
			\includegraphics[scale=.6]{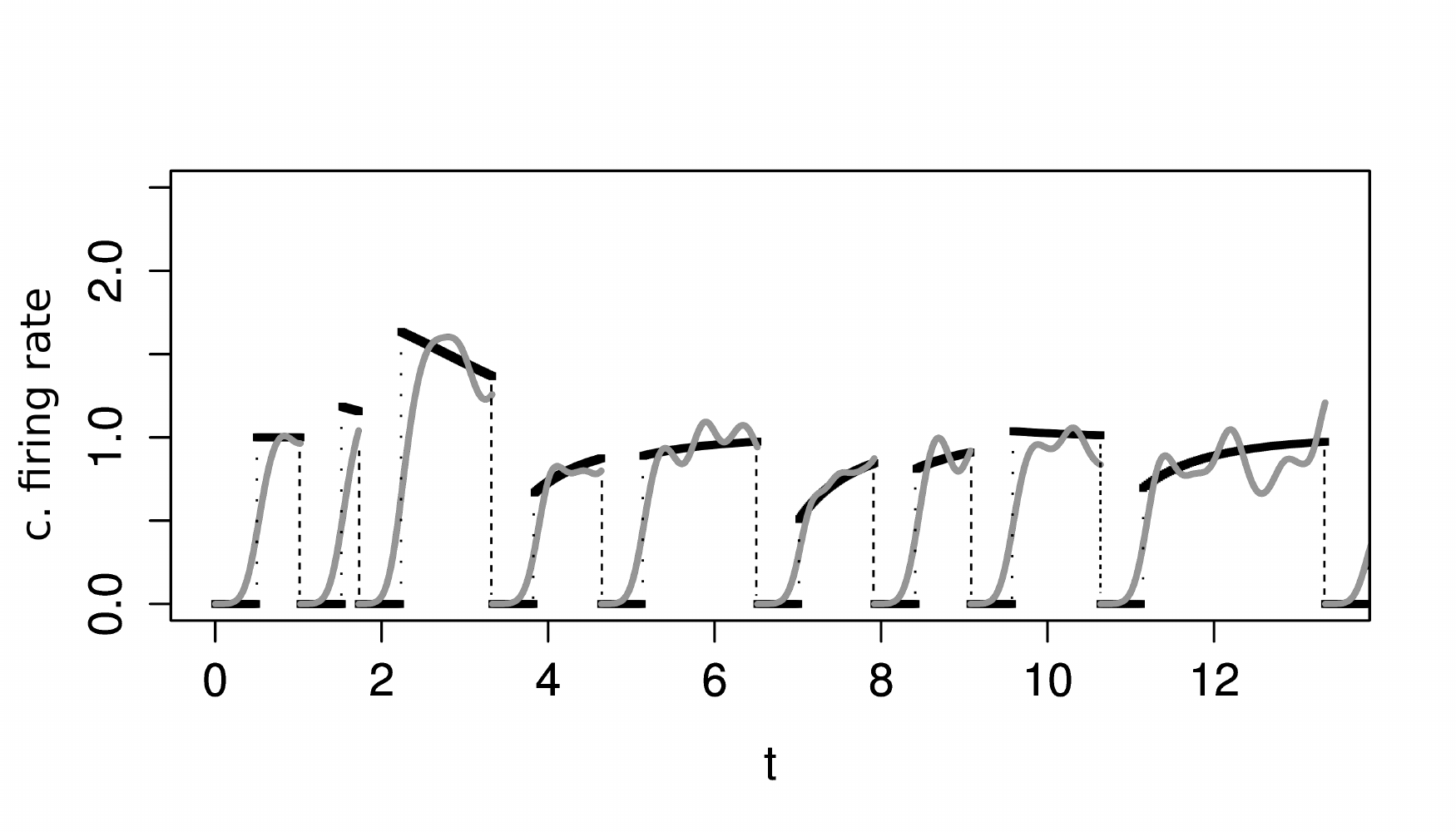}			
			\caption{\label{figfig}Two different realizations of the theoretical conditional instantaneous firing rate \eqref{norn}
			(black line) together with its estimate \eqref{hazardc} (grey line). The parameters are $(\lambda,\delta)=(1,0.5)$.}
		\end{figure}		
		Note furthermore that the use of \eqref{hazardl} to estimate \eqref{hazard1a} in \eqref{lambda}
		is an approximation because the sampling does not start at time zero of the theoretical model.
		
		We have the following theorem.	
	
		\begin{theo}
			\label{hazarda} 
			Let us consider a simple point process $N(t)$ (i.e.\ a point process admitting at most one single event at any time),
			with Markov, ergodic and stationary inter-event intervals
			$T_i$, $i\geq 1$. Under some regularity conditions on the ISI probability density functions and the kernels $K_1$ and $K_2$, for all
			sequences $\lbrace b_n\rbrace$ satisfying
			\begin{equation}
				\lim_{n\rightarrow +\infty}\frac{nb_n^4}{\ln n}=\infty,
			\end{equation}
			(see Appendix \ref{baba} for complete details)
			the estimators (\ref{hazardl}) and (\ref{hazardc}) are uniform strong consistent estimators of the hazard rate functions
			(\ref{hazard1a}) and
			(\ref{h2}) on $[0,M]$, $0<M\le L$, i.e.\
			\begin{equation}
				\lim_{n\rightarrow+\infty}\sup_{t\in[0,M]}\left|\hat{h}_n(t)-h(t)\right|=0 \qquad \text{a.s.}
			\end{equation}
			\begin{equation}
				\lim_{n\rightarrow+\infty}\sup_{(\tau,t)\in[0,M]^2}\left|\hat{h}_n(t| \tau)-h(t| \tau)\right|=0 \qquad \text{a.s.}
			\end{equation}
		\end{theo}
	
		First, note that since on each observed bin it is impossible to record more than a single spike, it is reasonable
		to model the spike train through a simple point process.
		Second, an estimator for $\lambda^* (t)$ can be given as
		\begin{align}
			\label{est}
			\hat{\lambda}^*_n (t) = \hat{h}_n(t-l_{i-1}|T_{i-1}=t_{i-1}), \qquad l_{i-1}<t\le l_i.
		\end{align}
		Formula \eqref{est} holds for $i=2,\dots, n$, because
		for an ergodic and stationary process, the estimation of $\lambda^*$ in the first ISI cannot be performed in the lack
		of knowledge on the preceding ISI.
		
		\begin{nb}
			The Markov, the ergodic and the stationary hypotheses concern the ISIs processes, i.e.\ the times $\{T_i\},i=1,2,\dots\:$.
			If the ISIs fulfill these hypotheses \eqref{16} holds. Note that the index set of the process $\{T_i\}$ is discrete and coincides with
			the positives integers. This hypothesis does not imply any stationary hypothesis on the process $N(t)$, $t>0$, that is an increasing
			process counting the number of spikes up to time $t$. Note that this last process has a continuous time index set. Despite these facts
			the two processes are related. Indeed it holds
			\begin{equation}
				\label{pro}
				\left\{N(t)<n\right\}=\left\{\sum_{i=1}^{n}T_i>t\right\}.
			\end{equation}		
		\end{nb}

	\section{An \emph{a posteriori} validation algorithm}
	
		\label{S5}
		The firing rate estimator \eqref{est} has good asymptotic properties when ISIs can be modeled through stationary ergodic processes.
		Having a sample of ISIs, specific tests of stationarity can be performed, checking the behavior of their mean or of higher moments
		or even using \emph{ad-hoc} techniques as in \citet{saczuc}.
		These checks cannot guarantee strong stationarity of ISIs but can help to detect time dependencies. More difficult is the check of ergodicity.
		This is a model feature that cannot be tested from a single time record of ISIs. Wishing to analyzing the reliability
		of the proposed estimator, we present here an \emph{a posteriori} validation algorithm. The idea of such procedure
		is the following. We know that, due to the time-rescaling theorem,
		any simple point process with an integrable conditional intensity function $\lambda^*(t)$ may be
		transformed into a unit-rate Poisson process, by the random time transformation 
		$t\rightarrow \Lambda^*(t)=\int_0^t \lambda^*(u)\,du$ (see Appendix \ref{resc} for the complete statement of the theorem).
		Hence, we perform such transformation and we check whether
		the obtained point process is a unit-rate Poisson process. If we cannot reject the Poisson characterization
		of the transformed process, we conclude that our estimator is reliable, having correctly acted
		as a firing rate.
	
		Formalizing, we first compute the conditional firing rate estimator $\hat{\lambda}_n^*$,
		defined by equation (\ref{est}). Then, we perform the following time transformation
		\begin{equation}
			\label{lavagna}
			t \mapsto \hat{\Lambda}_n^*(t)=\int_0^t \hat{\lambda}_n^*(u)\, du . 
		\end{equation}
		Under this time-change, the ISIs $T_i=l_i-l_{i-1}$, $i\geq 1$ and $l_0=0$, become
		\begin{equation}
			\label{tildeT}
			\tilde{T}_i=\hat{\Lambda}_n^*(l_i)-\hat{\Lambda}_n^*(l_{i-1})=\int_{l_{i-1}}^{l_i}\hat{\lambda}_n^*(u)\,du, \qquad i=1,\ldots,n.
		\end{equation}
		If the firing rate estimator (\ref{est}) is reliable, i.e.\ if the hypotheses supporting its computation are verified,
		the transformed ISIs (\ref{tildeT}) should be independent and identically distributed (i.i.d.)
		exponential random variables with mean 1, according to the time-rescaling theorem. 
		Hence, a way to validate our estimator on sample data is to check the independence and exponential distribution of the transformed ISIs.
		
		We can perform a goodness-of-fit test to verify whether the sequence $\lbrace \tilde{T}_i, i=1,\dots,n\rbrace$
		follows the exponential distribution with mean 1. Then we can compute any dependence index, like the correlation coefficient $\rho$
		\citep{Spearman} or the Kendall's tau \citep{Kendall}, of the couples $(\tilde{T}_{i},\tilde{T}_{i+1})$, $i=1,\ldots,n-1$,
		to check whether the successive transformed ISIs are independent.
		
		Here we propose also another alternative test of the hypotheses on the transformed ISIs, based on the concept of independent copula.
		Intuitively the independent copula is the joint probability distribution of two independent uniform random variables.
		Any group of independent random variables is characterized by the independent copula
		(see Appendix \ref{cop}).
		Having a sample of random variables following a distribution $F_X(x)$ it can be transformed into a sample of uniform random variables
		through the transformation $Y=F_X^{-1}(X)$ \citep{Nelsen}. 
		When our estimation of $\lambda^*$ is reliable, the sample obtained through transformation \eqref{lavagna} should contain independent
		exponential random variables of parameter $1$. In particular it should be characterized by the independent copula.		
		To verify this, we consider the further ISI transformation,
		based on the exponential distribution assumption,
		\begin{equation}
			\label{uniform}
			Z_i=1-e^{-\tilde{T}_i}.
		\end{equation}
		When $T_i \sim \exp(1)$, $\lbrace Z_i, i\geq1 \rbrace$ is a collection of i.i.d.\ uniform random variables on $[0,1]$.
		Therefore, the copula of $Z_i$ and $Z_{i+1}$, $i=1,2,\ldots,n$, should be the independent copula with uniform marginals. Hence,
		another way to verify the reliability of \eqref{est} on the transformed ISIs is to perform a uniformity test
		(goodness-of-fit test for the uniform distribution on $[0,1]$) and a goodness-of-fit test for the independent copula
		on the collection $\lbrace Z_i, i\geq1 \rbrace$. In particular, we apply the goodness-of-fit test for copulas proposed in \citet{Genst}.
		
		See Appendix \ref{lavagna2} for a description of the steps of the proposed algorithm.

	\section{The stochastic two-compartment neural model}
	
		\label{S6}
		Various first order Markov models could be used to exemplify the proposed technique (see \citet{Law} and the papers cited therein).
		
		We consider here the stochastic two-compartment model proposed in \citet{Lansky}. It is a two-dimensional Leaky Integrate and Fire type
		model that generalizes the one-dimensional Ornstein--Uhlenbeck model \citep{ric}.
		The two-compartment, for specific choices of the parameters, can exhibit independence or Markov property of different orders.
		The rationale for choosing this model lies in the fact that we wanted to deal with samples eventually not verifying assumptions supporting
		the reliability of our estimator.
		In fact, this example gives us the opportunity to illustrate the use of our ``a posteriori validation algorithm''.

		The introduction of the second component allows us to relax the
		resetting rule of the membrane potential.
		In a recent paper we showed that
		spike trains generated by this model exhibit dependent ISIs for suitable
		ranges of the parameters \citep{BenSac}. Here we use ISIs generated by
		this model to estimate the conditional firing rate through the estimator
		proposed in this paper.

		In the stochastic two compartment model the neuron is described by two
		interconnected compartments, modeling the membrane potential dynamics of
		the dendritic tree and the soma, respectively.
		The dendritic component is responsible for receiving external inputs, while the somatic component emits outputs.
		Hence, external inputs reach indirectly the soma by the interconnection between the two compartments. 
	
		The model is then described by a bivariate diffusion process $\lbrace\mathbf{X}(t), t\geq 0\rbrace$, whose components $X_1(t)$ and $X_2(t)$
		model the membrane potential evolution in the dendritic and somatic compartment, respectively. Assuming that external inputs have intensity $\mu$
		and variance $\sigma^2$, the process obeys to the system of stochastic differential equations
		\begin{equation}
			\label{bicomp}
			\begin{cases}
				dX_1(t) = \lbrace -\alpha X_1(t)+\alpha_r\left[X_2(t)-X_1(t)\right]+\mu\rbrace dt+\sigma dB(t), \vspace*{3mm}\\
				dX_2(t) = \lbrace -\alpha X_2(t)+\alpha_r\left[X_1(t)-X_2(t)\right]\rbrace dt.
			\end{cases}
		\end{equation}
		Here $\alpha_r$ accounts for the strength of the interconnection between the two compartments, while $\alpha$ is the leakage
		constant that models the spontaneous membrane potential decay in absence of inputs.
	
		\begin{figure}
			\centering
			\includegraphics[scale=.53]{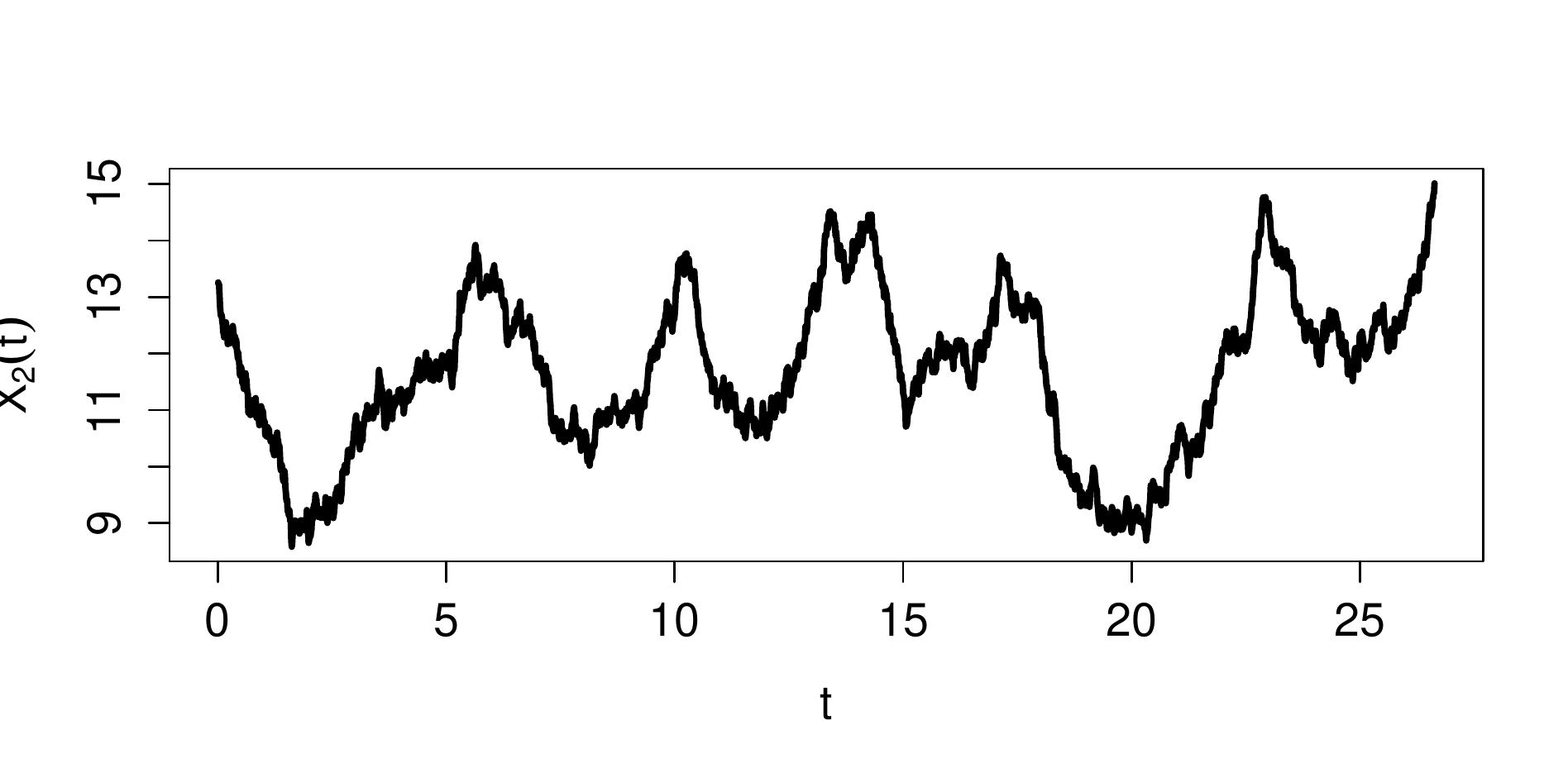} \\ \vspace{-1.2cm}
			\includegraphics[scale=.53]{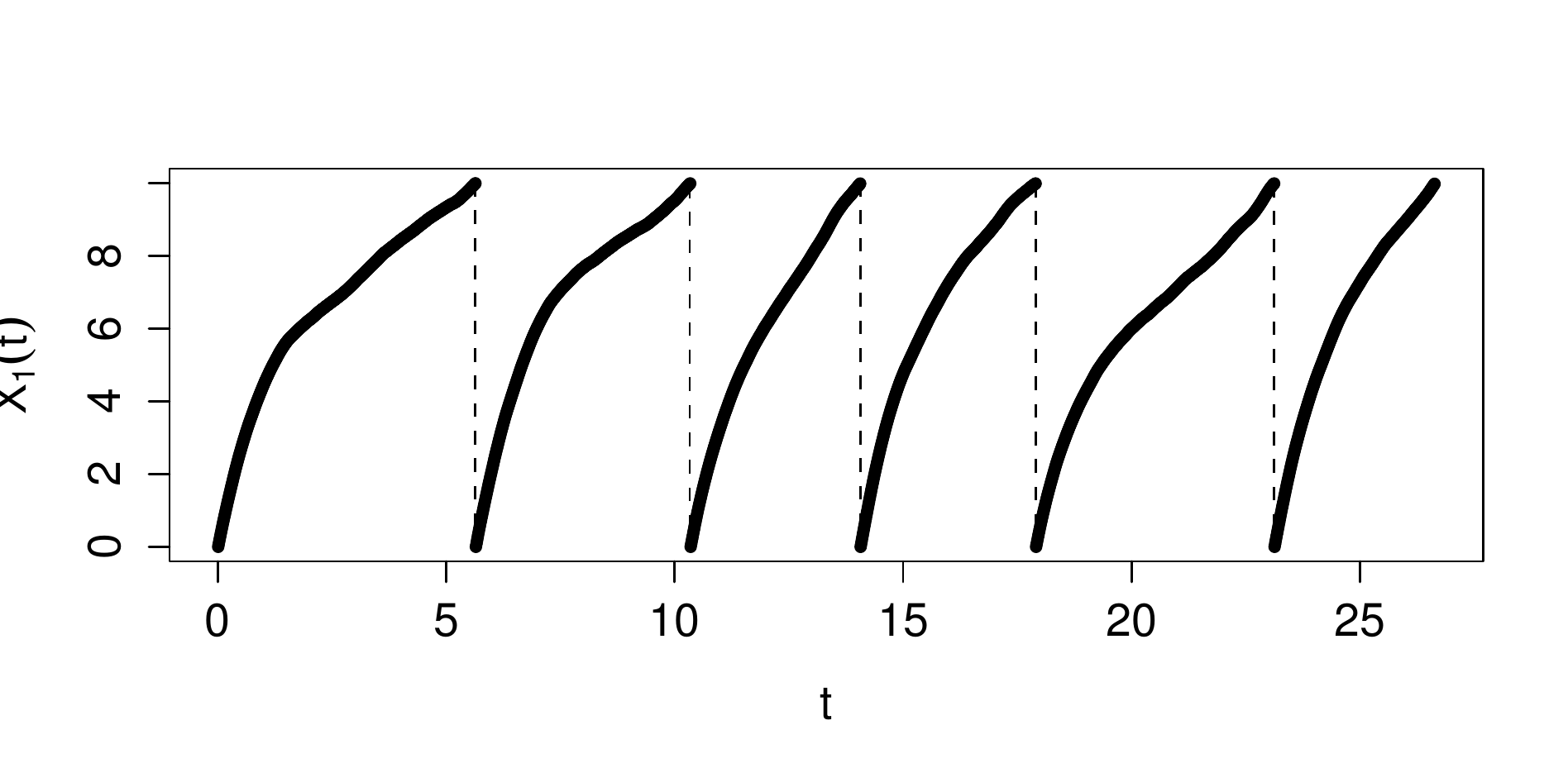} \\ \vspace{-1.2cm}
			\includegraphics[scale=.53]{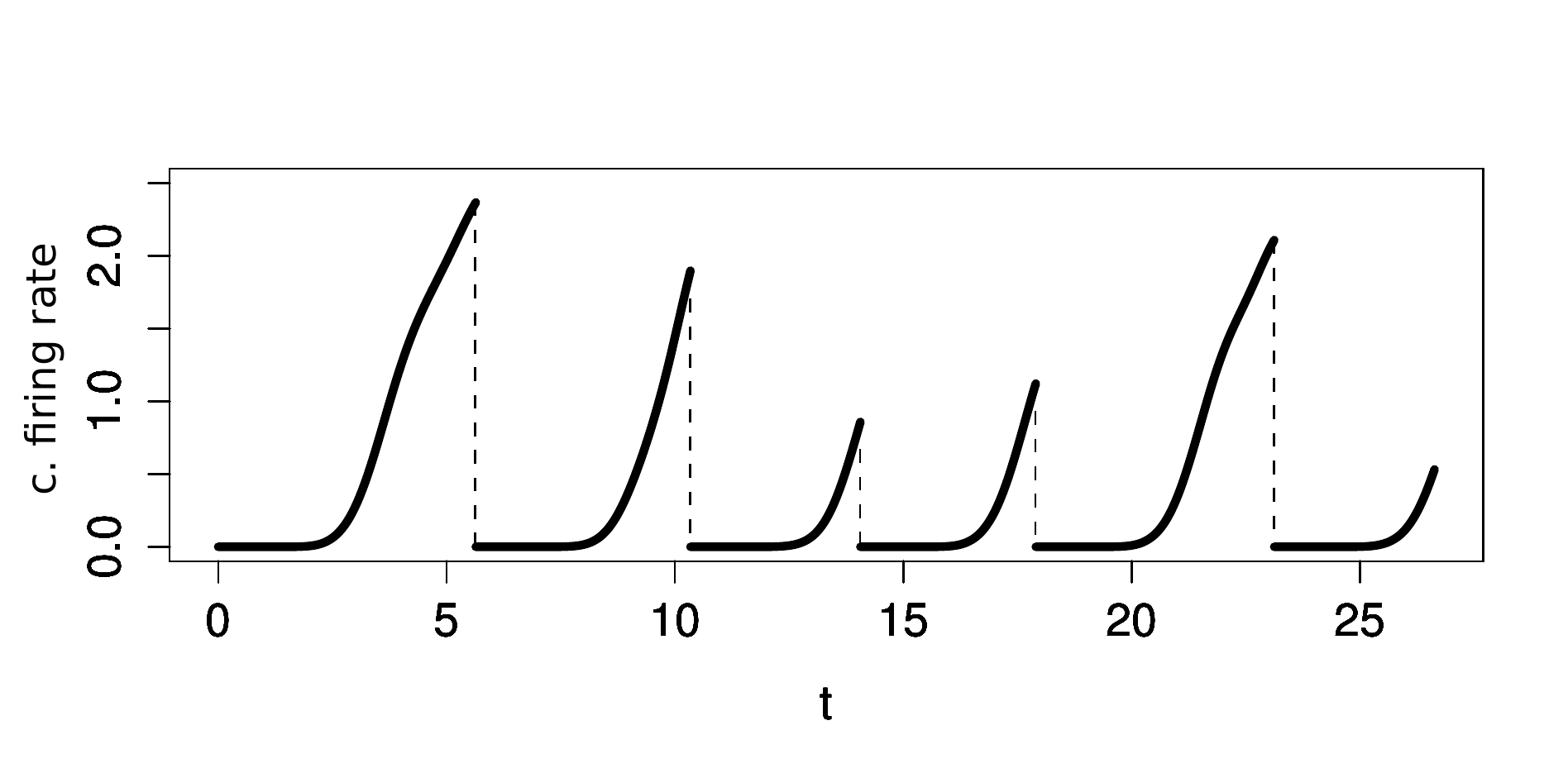}
			\caption{An example of the estimated conditional instantaneous firing rate for the two-compartment model described in Section \ref{S6}.
				In the upper and middle figures, the membrane potential respectively in the dendritic and somatic compartment is depicted.
				The lower figure shows the corresponding estimated conditional instantaneous firing rate $\hat{\lambda}^*(t)$.
				Here the parameters are set to $(\alpha,\alpha_r,\mu,\sigma,S)=(0.05,0.5,3.5,1,10)$. The estimation is performed
				by means of Gaussian kernels with standard deviation set at $1.5$. The ISIs here are rather short
				and the instantaneous firing rate rather high (but unknown). This can be understood
				by noticing that the conditional instantaneous firing rate is still growing when a new spike (and therefore
				a new conditioning event) occurs. The faster the instantaneous firing rate is, the sooner a conditioning event happens
				but still the presence of dependence between consecutive ISIs can be appreciated by considering
				the fact that after each recorded ISI the conditional instantaneous firing rate grows differently.}
		\end{figure}
		Whenever $X_2(t)$ reaches a characteristic threshold $S$ the neuron emits an action potential. Then $X_2(t)$ restarts
		from a resetting potential value that we assume equal to zero (a shift makes always possible this choice) while $X_1(t)$ continues its evolution.
		The absence of dendritic resetting makes the ISIs dependent on the past evolution of the neuron dynamics.
		Sensitivity analysis was applied in \citet{BenSac} to determine the
		model values of parameters that make its ISIs stationary but dependent. Here
		we simulate samples of 1000 ISIs from this model for different choices of
		the parameters, we determine the conditional firing rate \eqref{lambda}
		corresponding to each choice and we check the reliability of the obtained
		results by means of the algorithm proposed in the previous section. According
		to the hypotheses of Theorem \ref{hazarda} (see Appendix \ref{baba}) for our study we use: 
		\begin{enumerate}
			\item Gaussian kernels $K_1$ and $K_2$ with mean zero and standard deviation equal to $0.2$;
			\item Kernel weights $b_n=n^{-\beta}$ where $n$ is the sample size and $\beta=0.2$.
		\end{enumerate}
		We report the results of the algorithm in Table \ref{tab2}. The first three cases in
		Table \ref{tab2} correspond to Markov dependent ISIs as shown in
		\citet{BenSac}. The proposed estimator cannot be safely applied in the
		fourth case of Table \ref{tab2}. Indeed the use of algorithm of Section \ref{S5} shows that
		the hypotheses for its use are not verified. In this case, successive
		ISIs exhibit longer memory, as suggested in \citet{BenSac}. A variant
		of the proposed estimator can be easily developed to account for this longer
		memory. For example one should consider a rate intensity
		conditioned on the two previous ISIs.
	
		\begin{table}[h!]
			\renewcommand{\arraystretch}{1.3}
			\caption{Results of the uniformity test and the copula goodness-of-fit test of Algorithm \ref{Alg}, applied on ISIs simulated by the
				two-compartment model (\ref{bicomp}). In the last case the copula goodness-of-fit test fails, as the ISI process is statistically
				a Markov process of order 2. Indeed for values of $\mu$ close to the threshold $S$, the ISIs become very short and the evolution of
				the two-compartment model is more dependent on its past history (see \citet{BenSac} for details).}
			\label{tab2}
			\begin{center}
				\begin{tabular}{|c|c|c|}
					\hline
					\multirow{2}{*}{\textbf{Parameters}} 
					& \textbf{Uniformity test} & \textbf{Copula goodness-of-fit test} \\
					& \textbf{p-value} & \textbf{p-value}\\
					\hline
					$\alpha=0.05$, $\alpha_r=0.5$, & \multirow{2}{*}{0.88} & \multirow{2}{*}{0.97} \\
					$\mu=4$, $\sigma=1$, $S=10$ & & \\
					\hline
					$\alpha=0.05$, $\alpha_r=0.25$, & \multirow{2}{*}{0.84} & \multirow{2}{*}{0.62} \\
					$\mu=4$, $\sigma=1$, $S=10$ & & \\
					\hline
					$\alpha=0.05$, $\alpha_r=0.5$, & \multirow{2}{*}{0.69} & \multirow{2}{*}{0.87} \\
					$\mu=3.5$, $\sigma=1$, $S=10$ & & \\
					\hline					
					$\alpha=0.05$, $\alpha_r=0.5$, & \multirow{2}{*}{0.21} & \multirow{2}{*}{0.01} \\
					$\mu=8$, $\sigma=1$, $S=10$ & & \\
					\hline
				\end{tabular}
			\end{center}
		\end{table}
	
	\section{Application to experimental data}
	
		\label{S7}	
		
		We analyze here experimental data taken from the Internet\footnote{\texttt{http://lcn.epfl.ch/}$\sim$
		\texttt{gerstner/QuantNeuronMod2007/}}. This data were recorded and discussed by \citet{rauch}
		and correspond to single electrode in vitro recordings from rat neocortical pyramidal cells under random in vivo-like current
		injection. We limit our discussion to the analysis of the dataset file
		\texttt{07-1.spiketimes}. The injected colored noise current (Ornstein--Uhlenbeck process) determines dependencies between ISIs. For the analyzed
		dataset, the colored noise has a mean of 320 pA and a standard deviation of 41 pA.
		According to \citet{rauch}, cells responded with an action potential activity with stationary statistics that could be sustained
		throughout long stimulation intervals.
		The Kendall's Tau test on adjacent ISIs reveals a negative dependence; indeed $\tau = -0.118$ and the independence
		test furnishes us with a p-value equal to $0.0001$.
		Note that the lack of independence may be due both to the colored noise and to causes internal to the neuron's dynamics.
		Due to this fact it becomes interesting to estimate the conditional instantaneous firing rate (see Fig.~\ref{datidati}).
		The \emph{a posteriori} test for ergodicity (see Algorithm~\ref{Alg}) confirms the reliability of our estimation. Indeed
		the uniformity test on the transformed ISIs gives a p-value equal to 0.92 and the independence of the time-changed
		ISIs is confirmed by a test on the Kendall's Tau ($\tau = -0.0219$, $\text{p-value} = 0.47$).
		Furthermore, the goodness-of-fit copula test for independence copula gives a p-value equal to $0.99$.
		
		\begin{figure}
			\centering
			\includegraphics[scale=.56]{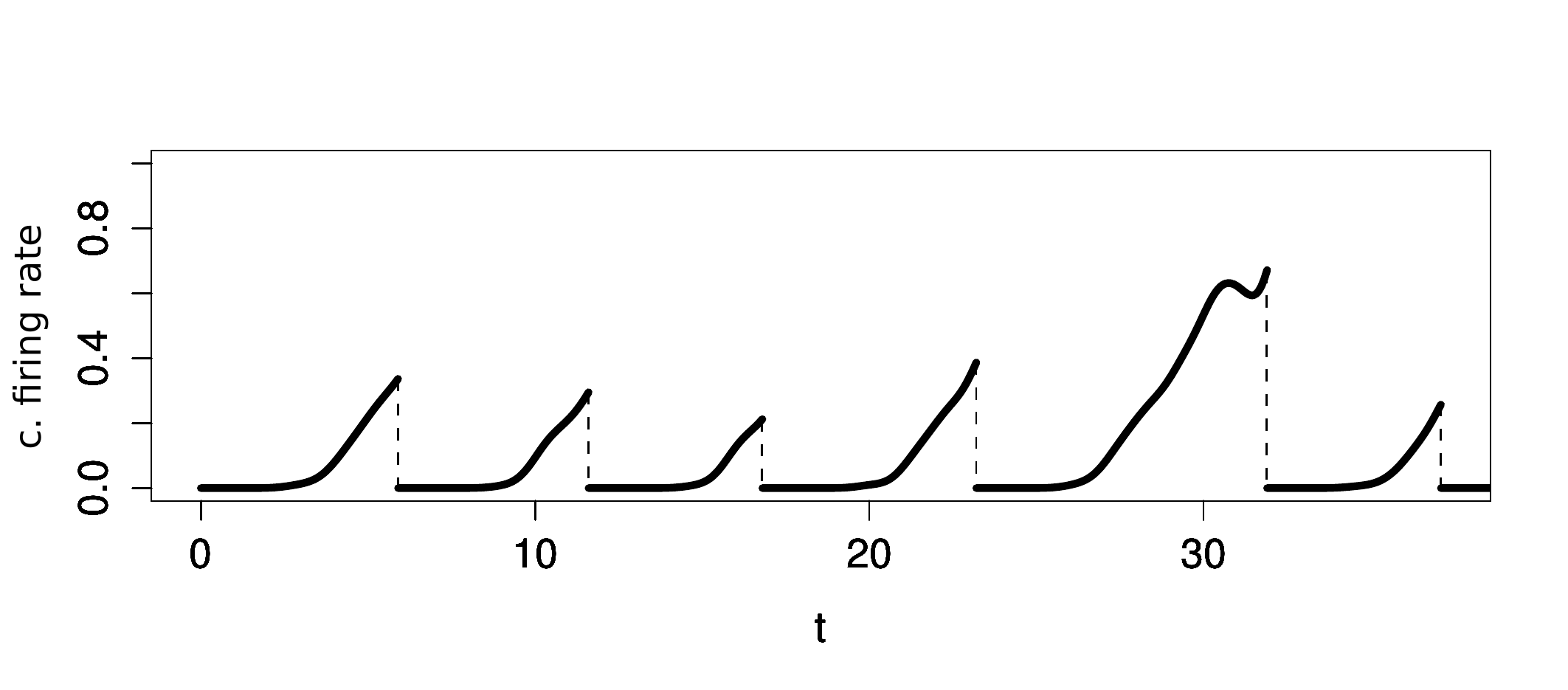}
			\caption{\label{datidati}An example of the estimated conditional instantaneous firing rate for real data as described in Section \ref{S7}.
			The estimation is performed with Gaussian kernels with standard deviation set at $1.5$. Time is measured in milliseconds.}
		\end{figure}

	\section{Conclusions}

		The analysis of the features exhibited by recorded ISIs often ground on the
		study of their statistics. The most used quantities in this framework are
		the firing rate and the instantaneous firing rate. Tipically its estimation
		is performed assuming that ISIs are independent despite experimental
		evidence on dependencies between successive ISIs. This choice may influence
		analysis results. Here we stressed consequences of ignoring dependencies on
		the the instantaneous rate estimation and we propose a non parametric
		estimator for the case of ergodic stationary ISIs with successive ISIs
		characterized by the Markov property. The use of such estimator could improve
		data analysis.
		
		We illustrated the use of the proposed estimator on simulated \ and on in
		vitro recorded data. Furthermore we proved its consistency property and we
		gave an algorithm to check the Markov ergodic stationary hypothesis on data.
		
		A natural extension of the proposed approach refers to data exhibiting
		memory also between non contiguous ISIs. Furthermore we plan to apply the
		method to non stationary data introducing their preliminary cleaning from
		possible periodicity or presence of trends.
	
	\subsection*{Acknowledgments}
	Work supported in part by University of Torino Grant 2012 \emph{Stochastic Processes and their Applications}
	and by project AMALFI (Universit\`a di Torino/Compagnia di San Paolo).

	\appendix
	
	\section{The time-rescaling theorem}
	
		\label{resc}
		For the sake of completeness we recall here the time-rescaling theorem that is a well-known result in probability theory.
		In neuroscience it is applied for example in \citet{Brown} and \citet{Gerhard} to develop goodness-of-fit tests for parametric point process
		models of neuronal spike trains.
	
		\begin{theo}[Time-rescaling theorem]
			\label{T-resc-theo}
			Let $N(t)$ be a point process admitting only unitary jumps, with integrable conditional intensity function
			$\lambda^*(t)$, modeling a spike train with interspike intervals $T_i=l_i-l_{i-1}$, $i\geq 1$, where $l_0=0$ 
			and $l_i$, $i\geq 1$, are the spiking epochs. We define
			\begin{equation}
				\label{Lambda}
				\Lambda^*(t)=\int_0^t \lambda^*(u)\, du.
			\end{equation}
			Then, under the random time transformation $t \mapsto \Lambda^*(t)$, the transformed process
			$\tilde{N}(t)=N({\Lambda^*}^{-1}(t))$ is a unit-rate Poisson process.
			
			Moreover the transformed ISIs
			\begin{align*}
				\tilde{T}_i=\Lambda^*(l_{i})-\Lambda^*(l_{i-1})=\int_{l_{i-1}}^{l_i}\lambda^*(u)du,
			\end{align*}
			are i.i.d.\ exponential random variables with mean 1, as they are waiting times of a unit-rate Poisson process. 
		\end{theo}

	\section{Bivariate copulas}
	
		\label{cop}
		Let us consider two uniform random variables $U$ and $V$ on $[0,1]$. Assume that they are not necessary independent. They are related
		by their joint distribution function
		\begin{equation}
			C(u,v) = \p(U\leq u,V\leq v).
		\end{equation}
		The function $C \colon [0,1]\times[0,1]\rightarrow[0,1]$ is called bivariate copula.
		
		Consider now the marginal cumulative distribution functions $F_X(x)=\p(X\leq x)$ and $F_Y(y)=\p(Y\leq y)$ of two random variables
		$X$ and $Y$. It is easy to check that 
		\begin{equation}
			\label{copula}
			C(F_X(x),F_Y(y))
		\end{equation}
		defines a bivariate distribution with marginals $F_X(x)$ and $F_Y(y)$.
		
		A celebrated theorem by \citet{sklar} establishes that any bivariate distribution can be written as in (\ref{copula}).
		Furthermore, if the marginals are continuous then the copula representation is unique \citep{Nelsen}.
		
		Copulas contain all the information related to the dependence between the random variables and do not involve marginal distributions.
		Hence, by means of copulas we can separate the study of bivariate distributions into two parts: the marginal behavior
		and the dependency structure between the random variables.
		The simplest bivariate copula is the independent copula. It is defined as 
		\begin{equation}
			\label{indcop}
			C(u,v) = uv, \qquad (u,v)\in[0,1]^2,
		\end{equation}
		and it represents the joint distribution of two independent uniform random variables on $[0,1]$.

	\section{Complete hypotheses and proof of Theorem \ref{hazarda}}
	
		\label{baba}
		Here we write
		\begin{align*}
			f_i(\tau,t|T_j,j=1,\dots,i-1)=
			\frac{d^2}{d\tau \,dt}\p(T_{i+1}\leq\tau,T_i\leq t|T_j,j=1,\dots,i-1)
		\end{align*}
		for the joint conditional density function of the couple $(T_i,T_{i+1})$ given the joint past history
		of an ISI and its subsequent until the couple $(T_{i-1},T_{i})$.
		The convergence properties of the estimator (\ref{hazardc}) depend on the following hypotheses. 
		\begin{enumerate}
			\item[h1.] The joint densities $f(\tau,t)$ and $f_i(\tau,t|T_j,j=1,\dots,i-1)$ belong to the space $C_0(\R^2)$ of real-valued continuous
				functions on $\R^2$ approaching zero at infinity.
			\item[h2.] The marginal densities $f(t)$ and $f_i(t|T_j,j=1,\dots,i-1)$ belong to the space $C_0(\R)$ of real-valued continuous
				functions on $\R$ approaching zero at infinity.
			\item[h3.] The conditional density functions $f_i(\tau,t|T_j,j=1,\dots,i-1)$ and $f_i(t|T_j,j=1,\dots,i-1)$ 
				are Lipschitz with ratio 1,
				\begin{align*}
					&\left| f_i(\tau,t|T_j,j=1,\dots,i-1)-f_i(\tau',t'|T_j,j=1,\dots,i-1)\right| 
					\leq |\tau-\tau'| + |t-t'|\mbox{ ,}\\
					&\left| f_i(t|T_j,j=1,\dots,i-1)-f_i(t'|T_j,j=1,\dots,i-1)\right| 
					\leq |t-t'|.
				\end{align*}
			\item[h4.] Let $[0,M]\subseteq\R_+$ be a compact interval. We assume that for all $t$ in an $\epsilon$-neighborhood
				$[0,M]_\epsilon$ of $[0,M]$ there exists $\gamma_{\epsilon} > 0$ such that $f(t) > \gamma_{\epsilon}$.
			\item[h5.] The kernels $K_j$, $j= 1, 2$, are H\"older with ratio $\mathfrak{L} < \infty$ and order $\gamma\in[0,1]$,
				\begin{align*}
					& \left|K_1(\tau) - K_1(\tau')\right|\leq \mathfrak{L} |\tau - \tau'|^{\gamma},
					\qquad (\tau, \tau') \in \R^2, \\
					& \left| K_2(t) - K_2(t')\right|\leq \mathfrak{L} |t - t'|^{\gamma},
					\qquad (t, t') \in \R^2.
				\end{align*}
		\end{enumerate}  
		
		\begin{nb}
			These assumptions are satisfied by any ergodic process with sufficiently smooth probability density functions (see \citet{Delacroix2}).
		\end{nb}
		
		To prove Theorem \ref{hazarda} we first recall
		that the estimator \eqref{kernel}
		is a uniform strong consistent estimator for $f(t)$ on any compact interval $[0,J]\subseteq\R_+$ \citep{parzen,rosenblatt}
		(see also \citet{Delacroix}), and then we need
		the following auxiliary lemma.
		\begin{lemma}
			\label{survival}
			Under the hypotheses and notations of Theorem \ref{hazarda},
			\begin{align}
				& \hat{S}_n(t)=1-\int_0^t\hat{f}_n(s)\,ds, \\
				& \hat{S}_n(t|\tau)=1-\int_0^t\hat{f}_n(s|\tau)\,ds,
			\end{align}
			are uniform strong consistent estimator on $[0,M]$ of the survival functions $S(t)=1-\int_0^t f(s)\,ds$
			and $S(t|\tau)=1-\int_0^t f(s|\tau)\,ds$, i.e.\
			\begin{align}
				& \lim_{n\rightarrow\infty}\sup_{t\in[0,M]}\left|S(t)-\hat{S}_n(t)\right|=0,\qquad \text{a.s.}\\
				& \label{arar} \lim_{n\rightarrow\infty}\sup_{t\in[0,M]}\left|S(t|\tau)-\hat{S}_n(t|\tau)\right|=0,\qquad \text{a.s.}
			\end{align}
		\end{lemma}
		
		\begin{proof}
			We report only the proof for $\hat{S}_n(t)$, as the proof for $\hat{S}_n(t|\tau)$ is analogous
			(see \citet{Ould-Said} and \citet{Arfi} for the necessary results on $\hat{f}_n(t|\tau)$).
			\begin{align*}
				\sup_{t\in[0,M]}\left|S(t)-\hat{S}_n(t)\right|
				& =   \sup_{t\in[0,M]}\left|\int_0^t\hat{f}_n(s)\,ds-\int_0^t f(s)\,ds\right|
				      \nonumber\\
				&\leq \sup_{t\in[0,M]}\int_0^t\left|\hat{f}_n(s)-f(s)\right|ds \nonumber\\
				&\leq \sup_{t\in[0,M]}\int_0^t\sup_{s'\in[0,M]}\left|\hat{f}_n(s')-f(s')\right|ds
				 	  \nonumber\\
				& =   \sup_{s'\in[0,M]}\left|\hat{f}_n(s')-f(s')\right|\sup_{t\in[0,M]}\int_0^t ds \nonumber\\
				& =  T\sup_{s'\in[0,M]}\left|\hat{f}_n(s')-f(s')\right| .
			\end{align*}
			Finally, by applying the convergence properties of the estimator (\ref{kernel}) (or of the estimator \eqref{cond-estimator}
			if we are concerned with $\hat{S}_n(t|\tau)$), we obtain the thesis. 
		\end{proof}
		
		\begin{nb}
			Observe that $S(t)$, $S(t|\tau)$, $\hat{S}_n(t)$ and $\hat{S}_n(t|	\tau)$ are bounded and strictly positive functions on $[0,M]$,
			as they are survival functions or sums of survival functions associated to strictly positive kernel functions.
		\end{nb}
		
		Now we have all the tools to prove Theorem \ref{hazarda}. 
		\begin{proof}[Proof of Theorem \ref{hazarda}]
			We report only the proof for the estimator $\hat{h}_n(t)$, as the proof for $\hat{h}_n(\tau|t)$ is analogous
			(see again \citet{Ould-Said} and \citet{Arfi} for the necessary results for the proof of the estimator
			of the conditional hazard rate function).
			Under the hypotheses h1--h5, we get
			\begin{align*}
				\sup_{t\in[0,M]}\left|\hat{h}_n(t)-h(t)\right|
				= {} & \sup_{t\in[0,M]}\left|\frac{\hat{f}_n(t)S(t)-f(t)\hat{S}_n(t)}
				{\hat{S}_n(t)S(t)}\right| \nonumber\\
				= {} & \sup_{t\in[0,M]}
				\left|\frac{\hat{f}_n(t)S(t)-f(t)S(t)+f(t)S(t)-f(t)\hat{S}_n(t)}
				{\hat{S}_n(t)S(t)}\right| \nonumber\\
				\leq {} & \sup_{t\in[0,M]}
				\frac{\left|\hat{f}_n(t)-f(t)\right|}		
				{\left|\hat{S}_n(t)S(t)\right|}\left| S(t) \right|
				+\sup_{t\in[0,M]}	
				\frac{\left| S(t)-\hat{S}_n(t)\right|}		
				{\left|\hat{S}_n(t)S(t)\right|}\left| f(t)\right| \\
				\le {} & \frac{\sup_{t\in[0,M]}\left|\hat{f}_n(t)-f(t)\right|}{\inf_{t\in[0,M]}\left|\hat{S}_n(t)S(t)\right|}
				\sup_{t\in[0,M]}\left| S(t) \right| \\
				& + \frac{\sup_{t\in[0,M]}\left|S(t)-\hat{S}_n(t)\right|}{\inf_{t\in[0,M]}\left|\hat{S}_n(t)S(t)\right|}
				\sup_{t\in[0,M]}\left| f(t) \right|
			\end{align*}
			Applying the convergence properties of the estimator (\ref{kernel}) and Lemma \ref{survival} we get the thesis, as $S(t)$
			and $\hat{S}_n(t)$ are bounded and strictly positive functions and $f(x)\in C_0(\R)$.
		\end{proof}
		
		\section{Details of the \emph{a posteriori} validation algorithm}
		
			\label{lavagna2}			
			Summarizing, an \emph{a posteriori} validation test for the proposed firing rate estimator (\ref{est}) follows this algorithm.
			\begin{algo}(\textit{Validation algorithm})
			\label{Alg}
			\begin{enumerate}
				\item Estimate the firing rate through (\ref{est}).  
				\item Compute the transformed ISIs
					\begin{align*}
						\tilde{T}_i = \int_{l_{i-1}}^{l_i}\hat{\lambda}_n^*(u)\,du, \qquad i=1,2,\ldots,n.
					\end{align*}
				\item Test the hypothesis that $\tilde{T}_i$ are i.i.d.\ exponential random variables of mean 1:
					\begin{enumerate}
						\item[a.] Apply the transformation
						$Z_i=1-e^{-\tilde{T}_i} \mbox{, } i=1,2,\ldots,n$, and test that these transformed random variables are uniform on $[0,1]$.
						\item[b.] Perform a goodness-of-fit test for the independent copula on the couples $(Z_i,Z_{i+1})$, $i=1,2,\ldots,n-1$. 
					\end{enumerate}
			\end{enumerate} 
		\end{algo}
		Note that testing that $Z_i$, $i=1,2,\ldots,n$, are i.i.d.\ uniform random variables on $[0,1]$ is equivalent to
		test that $\tilde{T}_i$,  $i=1,2,\ldots,n$, are i.i.d.\ exponential random variables with mean 1, i.e.\ that
		the transformed process is a unit-rate Poisson process.
	
		\begin{nb}
			In Algorithm \ref{Alg} we perform a goodness-of-fit test for copulas to verify the independence between $Z_i$ and $Z_{i+1}$.
			However any other test of independence can be applied, like a classical chi-squared test.
		\end{nb}

\end{document}